\documentclass[oneside,english]{amsart}
\usepackage[T1]{fontenc}
\usepackage[latin9]{inputenc}
\usepackage{amstext}
\usepackage{amsthm}
\usepackage{amssymb}
\usepackage[all]{xy}

\makeatletter
\numberwithin{equation}{section}
\numberwithin{figure}{section}
\theoremstyle{plain}
\newtheorem{thm}{\protect\theoremname}
\theoremstyle{plain}
\newtheorem{cor}[thm]{\protect\corollaryname}
\theoremstyle{definition}
\newtheorem{example}[thm]{\protect\examplename}
\theoremstyle{plain}
\newtheorem{lem}[thm]{\protect\lemmaname}
\theoremstyle{remark}
\newtheorem{rem}[thm]{\protect\remarkname}

\usepackage{amsrefs}
\usepackage{comment}

\newenvironment{thmprime}
{
  \edef\thmlabel{\thethm}
  \renewcommand{\thethm}{\thmlabel$'$}%
  \addtocounter{thm}{-1}%
   \begin{thm}}
  {\end{thm}}

\makeatother

\usepackage{babel}
\providecommand{\corollaryname}{Corollary}
\providecommand{\examplename}{Example}
\providecommand{\lemmaname}{Lemma}
\providecommand{\remarkname}{Remark}
\providecommand{\theoremname}{Theorem}

\begin{document}
\title{A note on sign of a self-dual representation}

\subjclass{20C33, 22E50}
\keywords{Frobenius-Schur indicator, 
Self-dual representation,
Deligne-Lusztig character
}

\author{manish mishra}
\address{Indian Institute of Science Education and Research Pune, Dr. Homi
Bhabha Road, Pasha, Pune 411008, Maharashtra, India}
\email{manish@iiserpune.ac.in}
\begin{abstract}
D. Prasad showed that the sign of a self-dual representation of a
finite or $p$-adic reductive group is often detected by a central
element. We study the extension of his results to some more general
situations and make some observations about the consequences of his
results.
\end{abstract}

\maketitle

\section{introduction}

Let $\mathcal{G}$ be a group and let $(\tau,V)$ be an irreducible
complex representation of $\mathcal{G}$. If $\tau$ is self-dual,
i.e., it is isomorphic to its contragradient $\tau^{\vee}$, then
there exists a non-degenerate $\mathcal{G}$-invariant bilinear form
$B:V\times V\rightarrow\mathbb{C}$ which is unique up to scalars.
It is thus either symmetric or skew symmetric. The \textit{sign} or
the \textit{Frobenius--Schur indicator} $\mathrm{sgn}(\tau)$ of
$\tau$ is defined to be $+1$ (resp. $-1$) according as $B$ is
symmetric (resp. skew-symmetric).

Let $G$ be a connected reductive group over a field $F$ where $F$
is either finite or non-archimedean local. In \cite{Pra98,Pra99},
D. Prasad showed that the sign of a self-dual representation $\pi$
of $G(F)$ is often detected by a certain order $\leq2$ element $\epsilon$
of the center $Z(F)$ of $G(F)$. This element $\epsilon$ can be
described in terms of half the sum of positive roots (see \S \ref{subsec:referee},
see also \cite{adams14}*{\S 5} for real groups). When $F$ is finite,
$Z$ is connected or finite with odd cardinality and $\pi$ is a generic
(i.e., a representation admitting a Whittaker model) self-dual representation
of $G(F)$, he showed \cite{Pra98} that $\mathrm{sgn}(\pi)$ is given
by the central character $\omega_{\pi}$ evaluated at $\epsilon$.
For an arbitrary group, it can be shown that $\epsilon$ fails to
detect the sign in general. 

In this note, we make some observations based on the results of Prasad.
For $F$ finite, in Theorem \ref{thm:pra'} we make a small improvement
to the Theorem of Prasad (Theorem \ref{thm:pra}), namely that it
works with a weaker hypothesis on $Z$. More specifically, if the
Frobenius co-invariants of the component group of $Z$ is of odd cardinality,
then $\epsilon$ detects the sign of an irreducible generic self-dual
representation. In Section \S \ref{subsec:referee}, Theorem \ref{thm:pra'}
is further refined to cover more groups. For an arbitrary finite reductive
group $G$, we show in Theorem \ref{thm:discon-cent} that the method
of Prasad of detecting $\mathrm{sgn}(\pi)$ can be made to work by
embedding $G$ into a reductive group $G'$ with connected center
and having the same derived group. This is illustrated more concretely
in Corollary \ref{lem:doscon-cent} in the case of an irreducible
generic, self-dual Deligne-Lusztig character $\pm R_{\mathtt{S}}^{\mathtt{G}}(\theta)$.

Corollaries \ref{cor:pra} and \ref{cor:pra2} observe another consequence
of Prasad's theorems, namely that for split reductive - finite or
$p$-adic groups - with connected center, the sign of an irreducible
generic self-dual representation is always $1$. 

\section{notations}

In \S 3, $\mathbb{F}_{q}$ denotes a finite field of order $q$ with
absolute Galois group $\Gamma$ and $\mathbb{F}$ denotes an algebraic
closure of $\mathbb{F}_{q}$. For an algebraic group $\mathcal{G}$
defined over a field $F$, we denote its identity component by $\mathcal{G}^{\circ}$
and its derived group by $\mathcal{G}_{\mathrm{der}}$. If $\tau$
is a representation of $\mathcal{G}(F)$, we denote its dual (or contragradient)
by $\tau^{\vee}$. If $\tau$ is self-dual, we denote its Frobenius-Schur
indicator by $\mathrm{sgn}(\pi)$. The central character of $\tau$
will be denoted by $\omega_{\tau}$. 

\section{\label{sec:finite}finite reductive group }

\subsection{\label{subsec:conn-cent}Groups with connected center}

Let $\mathtt{G}$ be a connected reductive group defined over $\mathbb{F}_{q}$
and let $\mathtt{Z}$ denote the center of $\mathtt{G}$. Let $\mathtt{B=TU}$
be an $\mathbb{F}_{q}$-Borel subgroup of $\mathtt{G}$, where $\mathtt{U}$
is the unipotent radical of $\mathtt{B}$ and $\mathtt{T}$ is an
$\mathbb{F}_{q}$-maximal torus of $\mathtt{G}$ contained in $\mathtt{B}$.
Denote by $X(\mathtt{T})$ (resp. $X^{\vee}(\mathtt{T})$), the character
(resp. co-character) lattice of $\mathtt{T}$ and by $\Phi$, the
set of roots of $\mathtt{T}$ in $\mathtt{G}$. Let $\Delta$ denote
the set of simple roots in $\Phi$ determined by $\mathtt{B}$. 

Recall that an irreducible representation $\pi$ of $\mathtt{G}(\mathbb{F}_{q})$
is called \textit{generic} if \\
$\mathrm{Hom}_{\mathtt{U}(\mathbb{F}_{q})}(\pi,\psi)\neq0$ for some
\textit{non-degenerate} character $\psi:\mathtt{U}(\mathbb{F}_{q})\rightarrow\mathbb{C}^{\times}$.
Here non-degenerate means that the stabilizer of $\psi$ in $(\mathtt{T}/\mathtt{Z})(\mathbb{F}_{q})$
is trivial. 

We now state a Theorem of D. Prasad on signs \cite{Pra98}*{Thm. 3}.
\begin{thm}
[Prasad] \label{thm:pra}If $\mathtt{Z}$ is connected or of odd
order, then there exists an element $s_{0}$ in $\mathtt{T}(\mathbb{F}_{q})$
such that it operates by $-1$ on all the simple root spaces of $\mathtt{U}$.
Further, $s_{0}^{2}$ belongs to $\mathtt{Z}(\mathbb{F}_{q})$ and
$s_{0}^{2}$ acts on an irreducible, generic, self-dual representation
by $1$ iff the representation is orthogonal. 
\end{thm}

Denote by $\sigma$ the Frobenius endomorphism of $\mathtt{G}$ and
by $(-)_{\sigma}$, the co-invariants with respect to $\sigma$. For
the existence of $s_{0}\in\mathtt{T}(\mathbb{F}_{q})$, we can weaken
the hypothesis on $\mathtt{Z}$ and rewrite Theorem \ref{thm:pra}
as,

\begin{thmprime}

\label{thm:pra'}Assume $(\mathtt{Z}/\mathtt{Z}^{\circ})_{\sigma}$
is of odd order. Then there exists an element $s_{0}$ in $\mathtt{T}(\mathbb{F}_{q})$
such that it operates by $-1$ on all the simple root spaces of $\mathtt{U}$.
Further, $s_{0}^{2}$ belongs to $\mathtt{Z}(\mathbb{F}_{q})$ and
$\mathrm{sgn}(\pi)=\omega_{\pi}(s_{0}^{2})$ for any irreducible,
generic, self-dual representation $\pi$ of $\mathtt{G}(\mathbb{F}_{q})$. 

\end{thmprime}

\begin{proof}

Let $\mathtt{T}_{\mathrm{ad}}$ denote the adjoint torus of $\mathrm{T}$.
From the short exact sequence
\[
\xymatrix{1\ar@{->}[r] & \mathtt{Z}\ar@{->}[r] & \mathtt{T}\ar@{->}[r] & \mathtt{T}_{\mathrm{ad}}\ar@{->}[r] & 1}
,
\]
we get the long exact sequence
\[
\xymatrix{1\ar@{->}[r] & \mathtt{Z}(\mathbb{F}_{q})\ar@{->}[r] & \mathtt{T}(\mathbb{F}_{q})\ar@{->}[r] & \mathtt{T}_{\mathrm{ad}}(\mathbb{F}_{q})\ar@{->}[r] & \mathrm{H}^{1}(\Gamma,\mathtt{Z})\ar@{->}[r] & \cdots}
.
\]

There is an element $\iota_{-}\in\text{\ensuremath{\mathtt{T}_{\mathrm{ad}}(\mathbb{F}_{q})}}$
which acts by $-1$ on all simple root spaces of $\mathtt{U}$. We
claim that $\iota_{-}$ admits a pull back in $\mathtt{T}(\mathbb{F}_{q})$.
To show that, it suffices to show that $\mathrm{H}^{1}(\Gamma,\mathtt{Z})$
is of odd cardinality. Now using the exact sequence 
\[
1\rightarrow\mathtt{Z}^{\circ}\rightarrow\mathtt{Z}\rightarrow\mathtt{Z}/\mathtt{Z}^{\circ}\rightarrow1,
\]
 and then taking the long exact sequence, we get that 
\[
\mathrm{H}^{1}(\Gamma,\mathtt{Z}^{\circ})\rightarrow\mathrm{H}^{1}(\Gamma,\mathtt{Z})\rightarrow\mathrm{H}^{1}(\Gamma,\mathtt{Z}/\mathtt{Z}^{\circ})
\]
is exact. By Lang's Theorem, $\mathrm{H}^{1}(\Gamma,\mathtt{Z}^{\circ})=0$.
Since $\mathrm{H}^{1}(\Gamma,\mathtt{Z}/\mathtt{Z}^{\circ})\cong(\mathtt{Z}/\mathtt{Z}^{\circ})_{\sigma}$,
the claim follows from our hypothesis. Now the result, as in the proof
of Theorem \ref{thm:pra}, follows as a corollary of \cite{Pra98}*{Lemma 1}.

\end{proof}

We record an easy corollary of Theorem \ref{thm:pra}. 
\begin{cor}
\label{cor:pra}If $\mathtt{G}$ is split, $\mathtt{Z}$ is connected
and $\pi$ is an irreducible, generic, self-dual representation, then
$\mathrm{sgn}(\pi)=1$. 
\end{cor}

\begin{proof}
Since $\mathtt{Z}$ is connected and $\mathtt{T}$ is split, there
is an $\mathbb{F}_{q}$-subtorus $\mathtt{Z}^{\mathrm{c}}$ of $\mathtt{T}$
such that $\mathtt{T}=\mathtt{Z}\times\mathtt{Z}^{\mathrm{c}}$. The
character $\omega_{\pi}$ of $\mathtt{Z}(\mathbb{F}_{q})$ then extends
trivially to a character $\theta$ of $\mathtt{T}(\mathbb{F}_{q})$.
Note that since $\pi$ is self-dual, $\omega_{\pi}^{2}=1$. Let $s_{0}=(z,z_{\mathrm{c}})\in\mathtt{T}(\mathbb{F}_{q})$
be as in Theorem \ref{thm:pra}. Since $s_{0}^{2}\in\mathtt{Z}(\mathbb{F}_{q})$,
we have $\omega_{\pi}(s_{0}^{2})=\theta(s_{0})^{2}=\omega_{\pi}(z)^{2}=1$. 
\end{proof}

\subsection{\label{subsec:referee}Further refinements of Theorem \ref{thm:pra'}}

In this subsection, we discuss cases where the sign of an irreducible,
generic, self-dual representation can be detected by a central element
even though $(\mathtt{Z}/\mathtt{Z}^{\circ})_{\sigma}$ is not necessarily
of odd order. All of these observations are due to the anonymous referee. 

Let $\lambda$ denote the sum of positive co-roots in $\Phi^{\vee}$
determined by $\mathtt{B}$. Choose $\zeta\in\mathbb{F}$ such that
$\zeta^{2}=-1$ and put $s=\lambda(\zeta)\in\mathtt{T}(\mathbb{F})$.
In other words, $s$ is the element 
\[
\chi\in X(\mathtt{T})\mapsto\zeta^{\langle\chi,\lambda\rangle}\in\mathbb{F}.
\]
 Then for all $\alpha\in\Delta$, $\alpha(s)=\zeta^{\langle\alpha,\lambda\rangle}=\zeta^{2}=-1.$
Thus $s^{2}\in\mathrm{ker}(\alpha)$ for all $\alpha\in\Delta$. Consequently
$\epsilon:=s^{2}\in\mathtt{Z}(\mathbb{F}_{q})$ which is of order
at most $2$. Now note that $\sigma(\lambda(\zeta))=\lambda(\zeta^{q})$.
So $s^{-1}\sigma(s)=\lambda(\zeta^{q-1})\in\mathtt{Z}(\mathbb{F}_{q})$
since $\zeta^{q-1}=\pm1$. Also note that if $\frac{1}{2}\lambda\in X^{\vee}(\mathtt{T)}$,
then $\zeta^{\langle\chi,\lambda\rangle}=-1^{\langle\chi,\frac{1}{2}\lambda\rangle}\in\mathbb{F}_{q}$
and thus $s=\sigma(s)$. 

We can thus further refine Theorem \ref{thm:pra'}.

\begin{thm}

For any irreducible, generic, self-dual representation $\pi$ of $\mathtt{G}(\mathbb{F}_{q})$,
$\mathrm{sgn}(\pi)=\omega_{\pi}(\epsilon)$ provided $s\in\mathtt{T}(\mathbb{F}_{q})$
and this happens if any of the following conditions hold:

\begin{enumerate}

\item[(a)]$(\mathtt{Z}/\mathtt{Z}^{\circ})_{\sigma}$ is of odd order.

\item[(b)]$q$ is even.

\item[(c)]$q\equiv1$ $\mathrm{mod}$ $4$.

\item[(d)]\label{cond_c}$\frac{1}{2}\lambda\in X^{\vee}(\mathtt{T})$. 

\end{enumerate}

\end{thm}
\begin{example}
Condition $(\mathrm{d})$ holds when
\begin{itemize}
\item $\mathtt{G}=\mathrm{SO}(n,\mathbb{F}_{q})$ is the special orthogonal
group.
\item $\mathtt{G}=\mathrm{Spin}(n,\mathbb{F}_{q})$ is a spin group and
$n\equiv0,\pm1$ $\mathrm{mod}$ $8$. 
\end{itemize}
\end{example}

Put $t=s^{-1}\sigma(s)$ and $\bar{t}$ the image of $t$ in $(\mathtt{Z}/\mathtt{Z}^{\circ})_{\sigma}$.
If $\bar{t}$ is trivial, then by Lang-Steinberg's Theorem, $t=z^{-1}\sigma(z)$
for some $z\in\mathtt{Z}(\mathbb{F})$. Put $r=sz^{-1}$. Then $r=\sigma(r)$,
so $r\in\mathtt{T}(\mathbb{F}_{q})$ and satisfies $\alpha(r)=-1$
for all $\alpha\in\Delta$. We conclude that if $\bar{t}$ is trivial,
then $r^{2}\in\mathtt{Z}(\mathbb{F}_{q})$ detects the sign of an
irreducible, generic self-dual representation of $\mathtt{G}(\mathbb{F}_{q})$. 

Note that the image of $t$ in $\mathtt{Z}/\mathtt{Z}^{\circ}$ has
order at most $2$. If $\mathtt{Z}/\mathtt{Z}^{\circ}$ is cyclic
of even order, then it has a unique element of order $2$. Therefore
in this case, $\bar{t}$ is trivial if and only if $(\sigma-1)(\mathtt{Z}/\mathtt{Z}^{\circ})$
has even order. Using this criterion, we obtain:

\begin{lem}

When $\mathtt{G}$ is simple and simply connected, the sign of an
irreducible, generic self-dual representation of $\mathtt{G}(\mathbb{F}_{q})$
is detected by $\epsilon$ if any of the conditions $(\mathrm{i})$
to $(\mathrm{vii})$ in \cite{TiZa04}*{Theorem 1.7} are satisfied.

\end{lem}

\subsection{\label{subsec:discon}Groups with disconnected center}

Let $\mathtt{G}$ be as before but without any connectedness assumption
on $\mathtt{Z}$ . Let $i:\text{\ensuremath{\mathtt{Z}}\ensuremath{\rightarrow\mathtt{Z}'} }$
be an embedding of $\mathtt{Z}$ into a torus $\mathtt{Z}'$ defined
over $\mathbb{F}_{q}$. Let $\mathtt{G}'$ denote the pushout of $i$
and the natural inclusion $\mathtt{Z}\hookrightarrow\mathtt{G}$.
Explicitly, let $\mathtt{G}'$ be the quotient of $\mathtt{G}\times\mathtt{Z}'$
by the closed normal subgroup $\{(z,z^{-1})\mid z\in\mathtt{Z}\}$.
Then the inclusions $\mathtt{G}\hookrightarrow\mathtt{G}^{\prime}$
and $\mathtt{Z}'\hookrightarrow\mathtt{G}'$ induced by $g\in\mathtt{G}\mapsto(g,1)\in\mathtt{G}\times\mathtt{Z}'$
and $z\in\mathtt{Z}'\mapsto(1,z)\in\mathtt{G}\times\mathtt{Z}'$ respectively
are $\sigma$-equivariant monomorphisms. The inclusion $\mathtt{G}\hookrightarrow\mathtt{G}^{\prime}$
is a \textit{regular embedding} in the sense of \cite{GM16}*{Def. 1.7.1}
by \cite{GM16}*{Lem. 1.7.3} . We identify $\mathtt{Z}'$ with its
image in $\mathtt{G}'$. Then $\mathtt{G}^{\prime}$ is a reductive
group over $\mathbb{F}_{q}$ with the following properties (see \cite{GM16}*{Remark 1.7.6, Lemma 1.7.7}
or \cite{leh78}*{\S 1}):

\begin{enumerate}

\item[(a)]The center of $\mathtt{G}'$ is $\mathtt{Z}'$ and $\mathtt{G}_{\mathrm{der}}^{\prime}=\mathtt{G}_{\mathrm{der}}$.
Also, $\mathtt{Z}=\mathtt{Z}^{\prime}\cap\mathtt{G}$ and $\mathtt{Z}(\mathbb{F}_{q})=\mathtt{Z}^{\prime}(\mathbb{F}_{q})\cap\mathtt{G}(\mathbb{F}_{q})$.

\item[(b)]We have a canonical exact sequence:
\[
1\rightarrow\mathtt{G}(\mathbb{F}_{q})\ldotp Z^{\prime}(\mathbb{F}_{q})\rightarrow\mathtt{G}^{\prime}(\mathbb{F}_{q})\rightarrow(\mathtt{Z}/\mathtt{Z}^{\circ})_{\sigma}\rightarrow1.
\]

\item[(c)]Let $\mathtt{S}$ be a maximal $\mathbb{F}_{q}$-torus
in $\mathtt{G}$. Then $\mathtt{S}^{\prime}=\mathtt{S}\ldotp\mathtt{Z}^{\prime}$
is a maximal $\mathbb{F}_{q}$-torus in $\mathtt{G}^{\prime}$ and
every maximal $\mathbb{F}_{q}$-torus of $\mathtt{G}^{\prime}$ is
of this form.

\end{enumerate}

Now let $(\pi,V)$ be an irreducible self-dual representation of $\mathtt{G}(\mathbb{F}_{q})$.
By (a) and (b) above the quotient $\mathtt{G}^{\prime}(\mathbb{F}_{q})/\mathtt{G}(\mathbb{F}_{q})$
is finite abelian. Therefore, there exists an irreducible representation
$\pi^{\prime}$ of $\mathtt{G}^{\prime}(\mathbb{F}_{q})$ whose restriction
to $\mathtt{G}(\mathbb{F}_{q})$ contains $\pi$ as a constituent.
The representation $\pi^{\prime}$ is unique up to a linear character
of $\mathtt{G}^{\prime}(\mathbb{F}_{q})$ which is trivial on $\mathtt{G}(\mathbb{F}_{q})$.
Since $\pi$ is self-dual, there is an isomorphism
\begin{equation}
f:\pi^{\prime}\overset{\sim}{\rightarrow}\pi^{\prime\vee}\nu^{-1}
\end{equation}
 for some $\nu\in\mathrm{Hom}(\mathtt{G}^{\prime}(\mathbb{F}_{q}),\mathbb{C}^{\times})$
trivial on $\mathtt{G}(\mathbb{F}_{q})$. This determines a non-degenerate
form 
\[
B_{f}:V\times V\rightarrow\mathbb{C},
\]
 on the space $V$ realizing $\pi^{\prime}$ satisfying
\[
B_{f}(\pi^{\prime}(g)u,\pi^{\prime}(g)v)=\nu^{-1}(g)B_{f}(u,v)
\]
for all $u,v\in V$ and $g\in\mathtt{G}^{\prime}(\mathbb{F}_{q})$.
The form $B_{f}$ is either symmetric or skew-symmetric and we write
$\mathrm{sgn}(\pi^{\prime})\in\{\pm1\}$ to be the sign of this form.
The restriction of $\pi^{\prime}$ to $\mathtt{G}(\mathbb{F}_{q})$
is multiplicity free \cite{GM16}*{Theorem 1.7.15} and from this it
easily follows \cite{kumar13}*{Lemma 4.15} that $\mathrm{sgn}(\pi)=\mathrm{sgn}(\pi^{\prime}).$
If $\pi$ is generic, it immediately follows that $\pi^{\prime}$
is also generic \cite{kumar13}*{Lemma 4.10}. 

Now let $\mathtt{T}'=\mathtt{TZ}'$. Then by (c) above, $\mathtt{T}'$
is a maximal $\mathbb{F}_{q}$-torus contained in $\mathtt{G}'$.
Since $\mathtt{Z}'$ is connected, there exists $s_{0}'\in\mathtt{T}'(\mathbb{F}_{q})$
such that it operates by $-1$ on each of the simple root spaces of
$\mathtt{U}$. Then the argument in the proof of \cite{Pra98}*{Lemma 1}
shows that 
\[
\mathrm{sgn}(\pi^{\prime})=\omega_{\pi^{\prime}}(s_{0}'^{2})\nu(s_{0}').
\]
We have proved,
\begin{thm}
\label{thm:discon-cent}Let $\pi$ be an irreducible, generic, self-dual
representation of $\mathtt{G}(\mathbb{F}_{q})$. Then $\mathrm{sgn}(\pi)=\omega_{\pi'}(s_{0}'^{2})\nu(s_{0}')$. 
\end{thm}

Now assume that $\pi$ corresponds to an irreducible generic self-dual
Deligne-Lusztig character $\pm R_{\mathtt{S}}^{\mathtt{G}}(\theta)$.
Let $\pi'$ to be an extension of $\pi$ corresponding to the irreducible
Deligne-Lusztig character $\pm R_{\mathtt{S}'}^{\mathtt{G}'}(\theta')$,
where $\mathtt{S}^{\prime}=\mathtt{S}\ldotp\mathtt{Z}^{\prime}$ and
$\theta'$ is an extension of $\theta$ to $\mathtt{S}'(\mathbb{F}_{q})$.
Let $W(\mathtt{G},\mathtt{S})$ (resp. $W(\mathtt{G}',\mathtt{S}')$)
denote the Weyl group of $\mathtt{G}$ (resp $\mathtt{G}'$). Note
that there is a $\sigma$-equivariant natural isomorphism $W(\mathtt{G},\mathtt{S})\cong W(\mathtt{G}',\mathtt{S}')$
\cite{GM16}*{Lemma 1.7.7}. Since $\text{\ensuremath{\pi} }$ is self-dual,
there is an element $w\in W(\mathtt{G},\mathtt{S})(\mathbb{F}_{q})$
which conjugates $\theta$ to $\theta^{-1}$. Write
\[
\mu=\theta'^{-1}\cdot({}^{w}\theta')^{-1}.
\]
Then $\mu$ is trivial on $\mathtt{S}(\mathbb{F}_{q})$. From \cite{GM16}*{Lemma 1.7.7, 1.7.8}
it follows that $\mathtt{G}^{\prime}(\mathbb{F}_{q})/\mathtt{G}(\mathbb{F}_{q})\cong\mathtt{S}'(\mathbb{F}_{q})/\mathtt{S}(\mathbb{F}_{q})$.
Therefore $\mu$ extends to a linear character $\nu$ of $\mathtt{G}'(\mathbb{F}_{q})$
which is trivial on $\mathtt{G}(\mathbb{F}_{q})$. Since $\omega_{\pi'}=\theta'|\mathtt{Z}'$,
we therefore get from Theorem \ref{thm:discon-cent},
\begin{cor}
\label{lem:doscon-cent}$\mathrm{sgn}(\pi)=\theta'(s_{0}'^{2})\nu(s_{0}')$. 
\end{cor}

\section{remarks on sign for p-adic groups}

Let $G$ be a quasi-split connected reductive group defined over a
non-archimedean local field $F$ and let $Z$ denote the center of
$G$. Let $B$ denote an $F$-Borel subgroup of $G$ with Levi factor
$T$ and unipotent radical $U$. 
\begin{lem}
\label{lem:t0}Assume $Z$ is an induced torus. Then there exists
an element $t_{0}$ of $T(F)$ such that it acts by $-1$ on all simple
root spaces of $U$ and satisfying $t_{0}^{2}\in Z(F)$. Consequently
for any irreducible generic self-dual representation $\pi$ of $G(F)$,
$\mathrm{sgn}(\pi)=\omega_{\pi}(t_{0}^{2})$. 
\end{lem}

\begin{proof}
We have an exact sequence $1\rightarrow Z\rightarrow T\rightarrow T_{\mathrm{ad}}\rightarrow1$,
where $T_{\mathrm{ad}}$ denotes the adjoint torus. Since $Z$ is
induced, by Hilbert's Theorem 90 and Shapiro's Lemma, $\mathrm{H}^{1}(F,Z)=1$.
Therefore we get an exact sequence
\[
1\rightarrow Z(F)\rightarrow T(F)\rightarrow T_{\mathrm{ad}}(F)\rightarrow1.
\]
 Now let $\iota_{-1}$ denote the element of $T_{\mathrm{ad}}(F)$
which acts by $-1$ on all simple root spaces of $U$ and choose $t_{0}$
to be a pullback of $\iota_{-}$ in $T(F)$. The result then follows
from \cite{Pra99}*{Prop. 2}. 
\end{proof}
\begin{cor}
\label{cor:pra2}Assume $G$ is split and $Z$ is connected. Let $\pi$
be an irreducible self-dual generic representation of $G(F)$. Then
$\mathrm{sgn}(\pi)=1$. 
\end{cor}

\begin{proof}
By Lemma \ref{lem:t0}, there is an element $t_{0}\in T(F)$ such
that $\mathrm{sgn}(\pi)=\omega_{\pi}(t_{0}^{2})$. Now the rest of
the argument is the same as in the proof of Corollary \ref{cor:pra}.
\end{proof}
\begin{rem}
Let $\rho^{\vee}$ denote half the sum of positive roots determined
by $B$ and let $\epsilon=2\rho^{\vee}(-1)$. Then $\epsilon\in Z(F)$.
If $Z^{\circ}$ is anisotropic and $\pi$ is an irreducible discrete
series representation of $G(F)$, then Conjecture 8.3 in \cite{GR10}
asserts that the sign $\mu(\pi)$ associated to the Deligne-Langlands
root number of $\pi$ is $\omega_{\pi}(\epsilon)$. Consequently if
$\pi$ is also self-dual and generic, then $\mathrm{sgn}(\pi)$ often
matches $\mu(\pi)$. 
\end{rem}

\section{acknowledgment}

The author thanks Alan Roche for pointing out mistakes in an earlier
draft of this article. The statements of Corollaries \ref{cor:pra}
and \ref{cor:pra2} are due to him. They were remarked to the author
during our correspondence. The author also thanks Dipendra Prasad
for many helpful discussions. Section \ref{subsec:discon} in this
note benefits from the ideas in \cite{kumar13}. 

The author is most thankful to the anonymous referee for his thorough
report. The contents of \S \ref{subsec:referee} are entirely due
to him.

\end{document}